\documentclass{article}
\usepackage{graphicx} 
\usepackage{amsmath,times,hyperref, multicol}
\usepackage{amsthm}
\usepackage{amssymb}
\usepackage{amsfonts}
\usepackage{latexsym}
\usepackage{graphicx}
\usepackage[usenames,dvipsnames]{color} 
\usepackage{enumitem}
\usepackage{halloweenmath}
\setlist[itemize]{leftmargin=*}
\setlist[enumerate]{leftmargin=*}
\usepackage[font={small}]{caption}
\usepackage{subcaption}
\usepackage{cite}
\usepackage{comment}
\usepackage{mathtools}
\mathtoolsset{showonlyrefs}

\newtheorem{Theorem}{Theorem}[section]
\newtheorem{Lemma}[Theorem]{Lemma}
\newtheorem{Proposition}[Theorem]{Proposition}

\newtheorem{Remark}[Theorem]{Remark}

\title{$q$-Numerical Ranges and Spectral Sets}
\author{
Ryan O'Loughlin\thanks{Department of Mathematics and Statistics, University of Reading,
Reading, England\\ \texttt{r.d.oloughlin@reading.ac.uk}}
\and
Jyoti Rani\thanks{Indian Institute of Science Education and Research (IISER) Mohali, Knowledge City, S.A.S Nagar, Punjab 140306\\ \texttt{jyotirani@iisermohali.ac.in}}
}
\date{}

\begin{document}

\maketitle

\begin{abstract}
We study spectral constants for convex domains $\Omega$ containing the spectrum of a matrix. We extend the Crouzeix–Palencia framework by obtaining bounds depending on a parameter $\gamma$ and relating these bounds to geometric properties of $\Omega$ and the numerical range $W(A)$. We generalise the proof that the numerical range is a $1+\sqrt{2}$-spectral set to scaled $q$-numerical ranges. We also formulate a generalisation of Crouzeix's Conjecture in the context of $q$-numerical ranges and give an application of this result to GMRES convergence bounds.
\end{abstract}

\medskip


\medskip


\section{Introduction}

Let $A$ be an $n \times n$ matrix and $\Omega \subseteq \mathbb{C}$ be a bounded convex domain with smooth boundary such that the spectrum of $A$ is contained in the interior of $ \Omega $. If for all $f \in {\cal A} ( \Omega )$ (holomorphic functions on the interior of $\Omega$ which are continuous on the boundary of $\Omega$) we have
$$
\| f(A) \| \leq K \sup_{z \in \Omega} |f(z)|
$$
then $\Omega$ is said to be a \textit{$K$-spectral set}. In this case we call $K$ the \textit{spectral constant} for $\Omega$ with respect to $A$. Since $\Omega$ is convex, Mergelyan's theorem shows that this is equivalent to requiring the same inequality for rational functions with poles outside the closure of $\Omega$.

One of the most important results in operator theory is von Neumann's inequality which states that when $\|A\| \leq 1$ the unit disk is a 1-spectral set. Crouzeix conjectured that the numerical range,
$W(A)=\{\langle Ax,x\rangle:\|x\|=1\}$, is a $2$-spectral set.
Crouzeix and Palencia proved in \cite{1root2} that the numerical range is a
$(1+\sqrt{2})$-spectral set. Since the first version of this paper, Crouzeix's Conjecture has been proved by Jin \cite{jin2026} and by Lorist and Schwenninger \cite{loristschwenninger2026}. Crouzeix's Conjecture has gained a lot of traction from the Functional and Complex Analysis communities \cite{o2024crouzeix, CC11.08, CC2matrix, CC3by3, CCgorkin, CCincreasing, CCLi, CCpalenciaextension, CCnilpotent, CCnum, malman2024double, 1root2, mashreghi2025schwarz}. The main result of this paper provides spectral constant bounds for the scaled
$q$-numerical range $\Omega_q$ defined in \eqref{eq}. These bounds
depend on the parameter $q$ and on a quantity measuring the non-normality of
$A$. When $q=1$ the result recovers the $1+\sqrt{2}$ spectral constant bound
for the numerical range obtained in \cite{1root2}.

In Section \ref{2} we extend previous results in \cite{CCpalenciaextension,crouzeix2019spectral}, and deduce a spectral constant for $\Omega$ which depends on a parameter $\gamma$ coming from a positivity criterion on the double layer potential and on a quantity depending on the distance between an extremal function and constants in a weighted norm. In Section \ref{3} we obtain explicit spectral constant bounds only dependent on the geometry of $\Omega$ and $W(A)$. In Section \ref{4} we specialise this to the case when $\Omega = \Omega_q$ is a scaled $q$-numerical range,
\begin{equation}\label{eq}
\Omega_q = \frac{W_q(A)}{q}, \quad W_q(A) = \{ \langle Ax,y \rangle : \|x\| = \|y\| = 1, \, \langle x,y \rangle = q \}, 
\end{equation}
where $q$ is a complex number such that $ 0< |q| \leq 1$ to deduce spectral constants for $\Omega_q$, which generalises the $1 + \sqrt{2}$ spectral constant result for $ W_1(A)$ in \cite{1root2}. Finally, we record the scaled $q$ analogue of Crouzeix's Conjecture proposed
in the first version of this paper, which has subsequently been proved in
\cite{aouichaouioloughlin2026} and independently in \cite{jin2026}, and give
an application of this result to GMRES convergence bounds.

Throughout, $A$ denotes an $n \times n$ matrix. Unless stated otherwise, $\Omega \subseteq \mathbb{C}$ denotes a smoothly bounded convex domain containing the spectrum of $A$ in its interior. This restriction entails no loss of generality since $X \subseteq \mathbb{C}$ is a $\kappa$-spectral set if and only if every smoothly bounded neighbourhood $\Omega$ of the spectrum of $A$ containing the closure of ${X}$ is also a $K$-spectral set \cite{schwenninger2025double}.

\section{Spectral Constants}\label{2}

Recall ${\cal A} ( \Omega )$ denotes the holomorphic functions on the interior of $\Omega$ which are continuos on the boundary of $\Omega$. For any function $f \in {\cal A} ( \Omega ) $, we define $f(A)$ via the Cauchy integral formula as
\[
f(A) = \frac{1}{2 \pi i} \int_{\partial \Omega} ( \sigma I - A )^{-1} f( \sigma )\,d \sigma .
\]
If $f(z) = \sum_{i=1}^k a_iz^i $ is a polynomial with $a_i \in \mathbb{C}$, then the functional calculus above becomes $f(A) = \sum_{i=1}^k a_iA^i$. If $|\partial \Omega| = L$ and $\sigma:  [0, L] \to \partial \Omega$ is an arc length parametrisation of $\partial \Omega$, then
\[
f(A) = \int_0^L \frac{\sigma' (s)}{2 \pi i} R( \sigma (s) , A ) f( \sigma (s))\,ds ,
\]
where $R( \sigma , A )$ is the resolvent, $( \sigma I - A )^{-1}$.

We will use the following terminology throughout. A function
$\hat f\in\mathcal A(\Omega)$ is called extremal for the pair $(A,\Omega)$ if
\[
\|\hat f\|_{\Omega}\leq 1
\qquad\text{and}\qquad
\|\hat f(A)\|
=
\sup_{\substack{f\in\mathcal A(\Omega)\\ \|f\|_{\Omega}\leq 1}}
\|f(A)\|.
\]
A unit vector $\hat x\in\mathbb C^n$ is called an associated extremal vector if
$\|\hat f(A)\hat x\|=\|\hat f(A)\|$.
Throughout, an associated extremal unit vector for $\hat f$ will be denoted by $\hat x$. Under our standing assumptions, the existence of an extremal function follows from \cite{CC2matrix}, while an associated extremal vector exists by finite
dimensionality. In particular, if $\|\hat f(A) \hat{x} \|\leq K$, then $\Omega$
is a $K$-spectral set for $A$.

A key idea in \cite{1root2} was to look at
\begin{equation}\label{cauchytransform}
g(A) := \int_0^L \frac{\sigma' (s)}{2 \pi i} R( \sigma (s) , A ) \overline{f( \sigma (s) )}\,ds,
\end{equation}
which in the scalar case reduces to the Cauchy transform of $\overline{f}$.
Crouzeix and Palencia analysed the operator
\begin{equation}\label{1st2}
     f(A) + g(A )^{*} = \int_0^L \mu ( \sigma (s) , A ) f( \sigma (s) )\,ds ,
\end{equation}
where the Hermitian matrix $\mu ( \sigma (s) , A )$ is
\begin{equation}
\mu ( \sigma (s) , A ) = \frac{\sigma' (s)}{2 \pi i} R( \sigma (s), A ) + \left[ 
\frac{\sigma' (s)}{2 \pi i} R( \sigma (s), A ) \right]^{*} . \label{mu}
\end{equation}
They showed that if $W(A) \subseteq \Omega$, then $\|f(A) + g(A)^*\| \leq 2$, and then consequently showed that the numerical range was a $1 + \sqrt{2}$-spectral set \cite{1root2} (see also \cite{ransford2018remarks} for an alternative shorter proof). In \cite{CCpalenciaextension}, the authors studied regions $\Omega$
that do not necessarily contain $W(A)$. They defined 
\begin{equation}\label{M}
M( \sigma , A ) := \mu ( \sigma , A ) - \lambda_{\min} ( \mu ( \sigma , A ) ) I 
\end{equation}
where $\lambda_{\min} ( \mu ( \sigma , A ))$ is the smallest eigenvalue of $\mu ( \sigma , A )$ at the point $\sigma$ on $\partial \Omega$.  Thus, by definition, 
$M( \sigma , A )$ is positive semidefinite on $\partial \Omega$.

We introduce the notation 
\begin{equation}\label{gamma}
    \gamma(f) := - \int_0^L \lambda_{\min}( \mu ( \sigma (s), A ) ) 
f( \sigma(s) )\,ds , \quad |\partial \Omega| = L .
\end{equation}

Using the same method of proof as in \cite{1root2}, the authors of \cite{CCpalenciaextension} proved the estimate in Lemma \ref{inequality} without the appearance of the \(\sqrt{1-S_f(x)^2}\) factor. A related abstract formulation was also obtained in \cite{SchwenningerDeVriesAbstract}. The proof of Lemma \ref{mainlemma} follows the same general ideas, with the additional ingredient of a dilation of the double-layer potential, which yields the improved factor \(\sqrt{1-S_f(x)^2}\).


For a continuous function $ u:\partial\Omega\to\mathbb C^n $, define the weighted seminorm by 
\[
\|u\|_M^2:=
\begin{cases}
\displaystyle
\frac{1}{2+\gamma(1)}
\int_{\partial\Omega}
\langle M(\sigma,A)u(\sigma),u(\sigma)\rangle\,ds,
& \text{if } M \neq 0,\\[2ex]
0,
& \text{if } M = 0.
\end{cases}
\]
The quantity $S_f(x)$ introduced below is the distance, with respect to $\|\cdot\|_M$, from $f(\cdot)x$ to the constant $\mathbb C^n$-valued functions on $\partial\Omega$.

\begin{Lemma}\label{mainlemma}
Let $\Omega$ be a smoothly bounded convex domain containing the spectrum of $A$ in its interior. For $f\in\mathcal A(\Omega)$ with
$\|f\|_\Omega=1$ and $\|x \| = 1$, set
\[
S_f(x):=\inf_{y\in\mathbb{C}^n}\|f(\cdot)x-y\|_M.
\]
Then for $g$ as defined in \eqref{cauchytransform}
\begin{equation}\label{inequality}
\|(f(A)+g(A)^*+\gamma(f)I)x\|
\leq
(2+\gamma(1))
\sqrt{1- S_f(x)^2}.
\end{equation}
\end{Lemma}

\begin{proof}
Positivity of $M( \sigma , A)$ as defined in \eqref{M} means $2+\gamma(1) \geq 0$ with equality if and only if $M \equiv 0$. In the equality case we have $f(A)+g(A)^*+\gamma(f)I=0$, so the asserted inequality is immediate. We therefore assume for the remainder of the proof that $2+\gamma(1)>0$ and $M\not\equiv 0$. Define
\[
\Phi(h)
=
\int_{\partial\Omega}h(\sigma)M(\sigma,A)\,ds,
\qquad h\in C(\partial\Omega).
\]
Since $M(\sigma,A)\geq 0$, the map $\Phi$ is positive. Thus $\Phi/(2+\gamma(1))$ is unital and positive, and hence
completely positive because $C(\partial\Omega)$ is commutative (Theorem 3.11 in \cite{paulsen2002completely}). By Stinespring's theorem \cite[Theorem~4.1]{paulsen2002completely},
there exist a Hilbert space $\mathcal K$, an isometry
$V:\mathbb{C}^n\to\mathcal K$, and a unital $*$-representation
$\pi:C(\partial\Omega)\to B(\mathcal K)$ such that
\begin{equation}\label{stinespring}
\frac{\Phi(h)}{2+\gamma(1)}=V^*\pi(h)V,
\qquad h\in C(\partial\Omega).
\end{equation}
Thus for $f \in \mathcal{A}( \Omega)$, $ x \in \mathbb{C}^n$
$$
\frac{\|(f(A)+g(A)^*+\gamma(f)I)x\|^2}
     {(2+\gamma(1))^2} = 
\left\| \frac{\Phi(f)x}{2+\gamma(1)} \right\|^2=\|V^*\pi(f)Vx\|^2,
$$
where the first equality follows from \eqref{1st2}, \eqref{M} and \eqref{gamma}.

For $y\in\mathbb C^n$, we have
\[
\begin{aligned}
\|f(\cdot)x-y\|_M^2
&=
\frac{1}{2+\gamma(1)}
\int_{\partial\Omega}
\langle M(\sigma,A)(f(\sigma)x-y),f(\sigma)x-y\rangle\,ds\\
&=
\frac{1}{2+\gamma(1)}
\Big(
\langle \Phi(|f|^2)x,x\rangle
-\langle \Phi(f)x,y\rangle\\
&\qquad\qquad
-\langle \Phi(\overline f)y,x\rangle
+\langle \Phi(1)y,y\rangle
\Big).
\end{aligned}
\]
Using \eqref{stinespring}, together with the fact that $V$ is an isometry and $\pi$ is a $*$-representation, this becomes
\[
\|\pi(f)Vx\|^2
-\langle \pi(f)Vx,Vy\rangle
-\langle Vy,\pi(f)Vx\rangle
+\|Vy\|^2 = 
\|\pi(f)Vx-Vy\|^2.
\]
Hence
\[
S_f(x)
=
\inf_{y\in\mathbb C^n}
\|\pi(f)Vx-Vy\|
=
\operatorname{dist}\bigl(\pi(f)Vx,V\mathbb C^n\bigr).
\]
Since $VV^*$ is the orthogonal projection onto $V\mathbb C^n$, it follows that
\[
S_f(x)=\|(I-VV^*)\pi(f)Vx\|.
\]
Orthogonal decomposition yields
\[
\begin{aligned}
\frac{\|[f(A)+g(A)^*+\gamma(f)I]x\|^2}
     {(2+\gamma(1))^2}
&=\|V^*\pi(f)Vx\|^2 \\
&=\|VV^*\pi(f)Vx\|^2 \\
&=\|\pi(f)Vx\|^2-\|(I-VV^*)\pi(f)Vx\|^2 \\
&\leq 1-S_f(x)^2,
\end{aligned}
\]
where the last inequality follows from
$\|\pi(f)\|\leq\|f\|_{\partial\Omega}=1$.




\end{proof}

\noindent In \cite{CCpalenciaextension}, the authors used the estimate obtained from Lemma~\ref{mainlemma} by omitting the pointwise factor $\sqrt{1-S_f(x)^2}$ to prove the following theorem.

\begin{Theorem}
Let $\Omega$ be a smoothly bounded convex domain containing the spectrum of $A$ in its interior.  Then
$\Omega$ is a $ 1 + \frac{\gamma(1)}{2} + \sqrt{2 + \gamma(1) + \gamma(1)^2 / 4 + m}$-spectral set where 
\begin{equation}
m = \int_0^L | \lambda_{\min} ( \mu ( \sigma (s) , A )) |\,ds .
\label{gamma(1)_hat}
\end{equation}
\label{theorem:1}
\end{Theorem}
\noindent Crouzeix and Greenbaum \cite{crouzeix2019spectral} obtained a related spectral constant bound which also applies to certain non-convex domains.


We make a pointwise sharpening of this result with the following theorem. The proof of the following theorem follows the argument of Schwenninger and de Vries in \cite{SchwenningerDeVriesAbstract}, with the pointwise refinement from Lemma~\ref{mainlemma} incorporated into their method.

\begin{Theorem}\label{main}
    \begin{enumerate}
        \item Let $\Omega$ be a smoothly bounded convex domain containing the spectrum of $A$ in its interior and let $\hat f$ be an extremal function with an associated extremal unit vector $\hat x$. Then $\Omega$ is a \begin{equation}\label{maineq}
    \max\left( 1, \left(1+\frac{\gamma(1)}{2}\right)
\sqrt{1-S_{\hat f}(\hat x)^2}
+
\sqrt{
\left(1+\frac{\gamma(1)}{2}\right)^2
\left(1-S_{\hat f}(\hat x)^2\right)+a(\Omega)
},\right)
\end{equation} 
-spectral set, where $S_{\hat{f}}$ is given by Lemma \ref{mainlemma}, $a(\Omega) = \sup_{f \in \mathcal{A}(\Omega)}\{ \inf_{\lambda \in \mathbb{C}}\|g- \lambda \|_{\infty , \Omega} :  \|f\|_{\infty, \Omega} \leq 1  \}$, where $g$ is defined as in \eqref{cauchytransform}.
\item There is a matrix $A$, a bounded, convex $\Omega$ with smooth boundary containing the spectrum of $A$ and an extremal function $\hat f$ with associated extremal unit vector $\hat x$ such that $$
\|\hat{f}(A)\| = \left(1+\frac{\gamma(1)}{2}\right)
\sqrt{1-S_{\hat f}(\hat x)^2}
+
\sqrt{
\left(1+\frac{\gamma(1)}{2}\right)^2
\left(1-S_{\hat f}(\hat x)^2\right)+a(\Omega)}.
$$
\end{enumerate}
\end{Theorem}
The second statement shows that the bound in part (1) is attained in general. In \cite{malman2024double} the quantity $a(\Omega)$ was studied to sharpen the $1+ \sqrt{2}$ spectral constant for the numerical range by an indefinite amount. Before proving Theorem \ref{main} we state a result, which will be instrumental in the proof of Theorem \ref{main}.

The following is Theorem 4.5 in \cite{CCgorkin}.

\begin{Theorem}\label{orthogthm}  Let $\Omega$ be a smoothly bounded convex domain containing the spectrum of $A$ in its interior. Let $\hat{f}$ be extremal for $(A,\Omega)$, and let $\hat{x}$ be an associated extremal unit vector. Then there exists a unique Borel probability measure $\mu$ on $\partial\Omega$ such that
\begin{equation}\label{1st}
\langle h(A)\hat x,\hat x\rangle
=
\int_{\partial\Omega}h\,d\mu \quad \text{ for all } h\in \mathcal{A}(\Omega).
\end{equation}
If, further, $\|\hat{f}(A)\| > 1$, then \begin{equation}\label{E:rep}
    \int_{\partial \Omega} \hat{f} d\mu = 0.
\end{equation}
\end{Theorem}

\noindent \textit{Proof of Theorem \ref{main}.}
    \begin{enumerate}
        \item Let $\hat{f}$ be an extremal function with associated extremal unit vector $\hat{x}$ and let $g$ be the Cauchy transform of $\overline{\hat{f}}$. If \(\|\hat f(A)\|\leq 1\), then the result is immediate. We therefore assume for the remainder of the proof that \(\|\hat f(A)\|>1\). By Lemma \ref{mainlemma} we have 
\begin{align}
    & \|\hat{f}(A)\|^2 + \inf_{\lambda \in \mathbb{C}} \left(\Re \int_{\partial\Omega}\hat{f}(g- \lambda) d \mu \right)  \\
    & = \|\hat{f}(A)\|^2 + \Re \int_{\partial\Omega}\hat{f}g d \mu \quad \text{by \eqref{E:rep}} \\
    &\leq |\|\hat{f}(A)\|^2 + \langle \hat{x}, \hat{f}g(A)\hat{x} \rangle| \quad \text{by \eqref{1st}}\\
    &=|\langle (\hat{f}(A)+g(A)^*+\gamma(\hat{f})I)\hat{x},\hat{f}(A)\hat{x} \rangle| \quad \text{by Lemma \ref{mainlemma} and \eqref{E:rep}} \\
    &\leq \left( (2+\gamma(1))
\sqrt{1- S_{\hat f}(\hat x)^2} \right) \|\hat{f}(A)\| \quad \text{by Lemma \ref{mainlemma}}.
\end{align}
Labelling
$
r=\inf_{\lambda\in\mathbb{C}}
\left(
\Re\int_{\partial\Omega}\hat{f}(g-\lambda)\,d\mu
\right),
$
algebraically manipulating the above yields
$$
\|\hat{f}(A)\|
\leq
\left(1+\frac{\gamma(1)}{2}\right)
\sqrt{1-S_{\hat f}(\hat x)^2}
+
\sqrt{
\left(1+\frac{\gamma(1)}{2}\right)^2
\left(1-S_{\hat f}(\hat x)^2\right)-r
}.
$$

Since \(\mu\) is a probability measure and \(\|\hat{f}\|_{\Omega}\leq1\), it follows that
$
|r|\leq a(\Omega),
$
and thus
$$
\|\hat{f}(A)\|
\leq
\left(1+\frac{\gamma(1)}{2}\right)
\sqrt{1-S_{\hat f}(\hat x)^2}
+
\sqrt{
\left(1+\frac{\gamma(1)}{2}\right)^2
\left(1-S_{\hat f}(\hat x)^2\right)+a(\Omega)
},
$$
and statement \((i)\) follows.

\item This statement follows by setting $A = \begin{pmatrix}
    0 & 2 \\
    0 & 0
\end{pmatrix}$, $ \Omega = W(A) = \{z \in \mathbb{C} :  |z| \leq 1 \}$, $f(z) = z$ and $x = (0, 1 ) ^T$. In this case, a direct computation yields $\lambda_{\min} ( \mu ( \sigma (s) , A )) = 0$ and $g = \overline{f(0)}$, and thus $a(\Omega) = 0$. Furthermore  Lemma~\ref{mainlemma} gives $2=\|f(A)x\|\leq 2\sqrt{1-S_f(x)^2}\leq 2$, we must have $f$ is extremal, and clearly $x$ is an associated extremal vector, and $S_f(x)=0$. 
\end{enumerate}
\qed

\begin{Remark}
Note that when $\Omega = W(A)$, after dropping the $S_{\hat{f}}$ improvement, we recover the spectral constant obtained in \cite{malman2024double}.
\end{Remark}

Although the main focus of the paper developed in later sections is on improving spectral constants through estimates of $\gamma(1)$, Lemma~\ref{mainlemma} also gives improvements through lower bounds on $S_{\hat f}(\hat{x})$. The following proposition shows that such a bound can be obtained directly from the distance interpretation of $S_{\hat f}(\hat{x})$. 

\begin{Proposition}\label{ellipse_S_bound}
Let
\[
A=
\begin{pmatrix}
1 & 2b\\
0 & -1
\end{pmatrix},
\qquad b>0,
\]
and let $\Omega=W(A)$. Let $\hat f$ be an extremal function for $(A,\Omega)$ and let $\hat{x}$ be an associated extremal unit vector. Then
\[
S_{\hat f}(\hat{x})
\geq
\frac{1}{(\sqrt{1+b^2}+b)^2}
\sqrt{\frac{b}{\sqrt{1+b^2}}}.
\]
\end{Proposition}

\begin{proof}
Write $a=\sqrt{1+b^2}$ and $\rho=a+b$. By \cite{mashreghi2025schwarz}, the extremal function $\hat f$ is an odd conformal map from $\Omega^\circ$ onto $\mathbb D$. Since $S_{\hat f}(\hat{x})$ is unchanged when $\hat f$ is multiplied by a unimodular constant, we assume $\hat f'(0)>0$. In particular $\hat f(A)=\hat f(1)A$ thus an extremal vector is
\[
\hat{x}=\frac{1}{\sqrt{1+\rho^2}}\binom{1}{\rho}.
\]

Parametrise $\partial\Omega$ by its Joukowski mapping
\[
\sigma(\zeta)=\frac12\left(\rho\zeta+\rho^{-1}\zeta^{-1}\right),
\qquad |\zeta|=1.
\]
A direct calculation gives
\begin{equation}\label{ellipse_mu}
\frac12\langle\mu(\sigma(\zeta),A)u,u\rangle\,ds
=
\frac{\rho^2-1}{|\rho-\zeta|^2}
\left|
u_1+\frac{\rho-\zeta}{\rho\zeta+1}u_2
\right|^2dm(\zeta),
\end{equation}
where $dm$ is normalized Lebesgue measure on the unit circle. Thus $\mu(\sigma(\zeta),A)$ is positive semidefinite of rank one, so $\lambda_{\min}(\mu(\sigma(\zeta),A))=0$. Hence $\gamma(1)=0$, $M=\mu$ and
\begin{equation}\label{ellipse_seminorm}
\|u\|_M^2
=
\int_{|\zeta|=1}
\frac{\rho^2-1}{|\rho-\zeta|^2}
\left|
u_1(\zeta)+
\frac{\rho-\zeta}{\rho\zeta+1}u_2(\zeta)
\right|^2dm(\zeta).
\end{equation}

Under the isometric identification in \eqref{ellipse_seminorm}, the constant $\mathbb C^2$-valued functions correspond to the span in $L^2(dm)$ of
\[
\frac{\sqrt{\rho^2-1}}{|\rho-\zeta|}
\quad\text{and}\quad
\frac{\sqrt{\rho^2-1}}{|\rho-\zeta|}
\frac{\rho-\zeta}{\rho\zeta+1}.
\]
Since $1+\rho(\rho-\zeta)/(\rho\zeta+1)=(1+\rho^2)/(\rho\zeta+1)$, the function $\hat f(\cdot)\hat{x}$ corresponds to
\begin{equation}\label{ellipse_fx}
\hat f(\sigma(\zeta))
\frac{\sqrt{\rho^4-1}}
{|\rho-\zeta|(\rho\zeta+1)}.
\end{equation}

Consider the unit function $\eta\in L^2(dm)$ defined by
\begin{equation}\label{ellipse_eta}
\overline{\eta(\zeta)}
=
\frac{\zeta(\rho\zeta+1)}{\rho+\zeta}
\frac{|\rho-\zeta|}{\rho-\zeta},
\qquad |\zeta|=1.
\end{equation}
Indeed, $|\eta(\zeta)|=1$, since $|\rho\zeta+1|=|\rho+\zeta|$ for $|\zeta|=1$. Using $dm(\zeta)=d\zeta/(2\pi i\zeta)$, we have
\begin{align*}
\int_{|\zeta|=1}
\frac{\sqrt{\rho^2-1}}{|\rho-\zeta|}
\overline{\eta(\zeta)}\,dm(\zeta)
&=
\frac{\sqrt{\rho^2-1}}{2\pi i}
\int_{|\zeta|=1}
\frac{\rho\zeta+1}{\rho^2-\zeta^2}\,d\zeta
=0,\\
\int_{|\zeta|=1}
\frac{\sqrt{\rho^2-1}}{|\rho-\zeta|}
\frac{\rho-\zeta}{\rho\zeta+1}
\overline{\eta(\zeta)}\,dm(\zeta)
&=
\frac{\sqrt{\rho^2-1}}{2\pi i}
\int_{|\zeta|=1}\frac{1}{\rho+\zeta}\,d\zeta
=0,
\end{align*}
because $\rho>1$ and the poles of both integrands lie outside the unit disk. Thus $\eta$ is orthogonal to the subspace corresponding to the constant $\mathbb C^2$-valued functions. It follows from the distance characterization of $S_{\hat f}(\hat{x})$ and \eqref{ellipse_fx} that
\begin{align}
S_{\hat f}(\hat{x})
&\geq
\left|
\int_{|\zeta|=1}
\hat f(\sigma(\zeta))
\frac{\sqrt{\rho^4-1}}
{|\rho-\zeta|(\rho\zeta+1)}
\overline{\eta(\zeta)}\,dm(\zeta)
\right| \notag\\
&=
\sqrt{1-\rho^{-4}}\,
\frac{1}{2\pi i}
\left|
\int_{|\zeta|=1}
\frac{\hat f(\sigma(\zeta))}
{1-\rho^{-2}\zeta^2}\,d\zeta
\right|.
\label{ellipse_distance}
\end{align}

The integrand in \eqref{ellipse_distance} is holomorphic in the annulus $\rho^{-1}<|\zeta|<1$ and continuous on its boundary, so the contour may be deformed to $|\zeta|=\rho^{-1}$. Since $\sigma(\rho^{-1}e^{i\theta})=\cos\theta$, setting $\zeta=\rho^{-1}e^{i\theta}$ gives
\begin{equation}\label{ellipse_contour}
\frac{1}{2\pi i}
\int_{|\zeta|=1}
\frac{\hat f(\sigma(\zeta))}
{1-\rho^{-2}\zeta^2}\,d\zeta
=
\frac{\rho^{-1}}{2\pi}
\int_0^{2\pi}
\frac{\hat f(\cos\theta)e^{i\theta}}
{1-\rho^{-4}e^{2i\theta}}\,d\theta.
\end{equation}

By \cite{mashreghi2025schwarz}, $\hat f$ is increasing and concave on $[0,a]$, with $\hat f(0)=0$ and $\hat f(a)=1$. Hence $\hat f(t)\geq t/a$ for $0\leq t\leq a$. Using the oddness of $\hat f$ therefore gives
\begin{equation}\label{ellipse_concavity}
\hat f(t)t\geq\frac{t^2}{a},
\qquad -1\leq t\leq1.
\end{equation}
Moreover,
\[
\operatorname{Re}
\left(
\frac{e^{i\theta}}{1-\rho^{-4}e^{2i\theta}}
\right)
=
\frac{(1-\rho^{-4})\cos\theta}
{1-2\rho^{-4}\cos(2\theta)+\rho^{-8}}.
\]
Taking real parts in \eqref{ellipse_contour} and applying \eqref{ellipse_concavity} gives
\begin{equation}\label{ellipse_integral}
\left|
\frac{1}{2\pi i}
\int_{|\zeta|=1}
\frac{\hat f(\sigma(\zeta))}
{1-\rho^{-2}\zeta^2}\,d\zeta
\right|
\geq
\frac{\rho^{-1}(1-\rho^{-4})}{2\pi a}
\int_0^{2\pi}
\frac{\cos^2\theta}
{1-2\rho^{-4}\cos(2\theta)+\rho^{-8}}\,d\theta.
\end{equation}
For $0<s<1$,
\[
\frac{1}{2\pi}
\int_0^{2\pi}
\frac{\cos^2\theta}
{1-2s\cos(2\theta)+s^2}\,d\theta
=
\frac{1}{2(1-s)},
\]
which follows from $\cos^2\theta=(1+\cos(2\theta))/2$ and the Poisson kernel identities
\[
\frac{1}{2\pi}\int_0^{2\pi}
\frac{d\theta}{1-2s\cos(2\theta)+s^2}
=
\frac{1}{1-s^2},
\qquad
\frac{1}{2\pi}\int_0^{2\pi}
\frac{\cos(2\theta)\,d\theta}
{1-2s\cos(2\theta)+s^2}
=
\frac{s}{1-s^2}.
\]
Applying this to \eqref{ellipse_integral} with $s=\rho^{-4}$ yields
\[
\left|
\frac{1}{2\pi i}
\int_{|\zeta|=1}
\frac{\hat f(\sigma(\zeta))}
{1-\rho^{-2}\zeta^2}\,d\zeta
\right|
\geq
\frac{1}{2a\rho}
=
\frac{1}{1+\rho^2},
\]
where $2a=\rho+\rho^{-1}$. Combining this with \eqref{ellipse_distance} gives
\[
S_{\hat f}(\hat{x})
\geq
\frac{\sqrt{1-\rho^{-4}}}{1+\rho^2}
=
\frac{1}{\rho^2}
\sqrt{\frac{\rho^2-1}{\rho^2+1}}.
\]
Finally, since $\rho=\sqrt{1+b^2}+b$ and $(\rho^2-1)/(\rho^2+1)=b/\sqrt{1+b^2}$, we obtain
\[
S_{\hat f}(\hat{x})
\geq
\frac{1}{(\sqrt{1+b^2}+b)^2}
\sqrt{\frac{b}{\sqrt{1+b^2}}},
\]
as required.
\end{proof}

\section{Geometric Bounds for Spectral Constants}\label{3}
From Theorem \ref{main} we see the spectral constant bound for $\Omega$ depends on the parameter $\gamma(1)$ (where $\gamma$ is defined in \eqref{gamma}) and $S_{\hat{f}}$. In this section we estimate the parameter $\gamma(1)$ to obtain spectral constants dependent on the geometry of $\Omega$ and $W(A)$. The following two lemmas are needed for the proof of Theorem \ref{maingamma}.

The function,
$$
h_{\Omega}(e^{i \theta}) := \sup_{p \in \Omega} \Re(pe^{-i \theta})
$$
is called the support function of $\Omega$.

\begin{Lemma}\label{per}
$$
\int_0^{2 \pi} h_{\Omega}( e ^{i \theta}) d \theta = |\partial \Omega|.
$$
\end{Lemma}
\begin{proof}
The projection of $\Omega$ on $e^{i \theta} \mathbb{R}$ has length $h_{\Omega}( e ^{i \theta})  + h_{\Omega}( -e ^{i \theta})$, so Cauchy's surface area formula (see, for example Theorem 5.2.1 in \cite{Schneider2014}) yields \begin{equation}\label{perimetersupport}
        \int_0^{2 \pi}  h_{\Omega}(e^{ i \theta})  d \theta = |\partial \Omega|.
    \end{equation}
\end{proof}

\begin{Lemma}\label{3.1}
\begin{enumerate}
    \item Suppose $W(A)$ has non-empty interior and let $ W(A) \subseteq \Omega$ and let $v$ be the unit normal at $\sigma \in \partial \Omega$. Then 
\begin{equation}\label{lambdaminforqnumrange}
      \frac{1}{\pi} r(\sigma) \left( h_{\Omega}(v) - h_{W(A)}(v) \right) \leq \lambda_{\min} ( \mu ( \sigma  , A )),
\end{equation}
where $r(\sigma) = \inf_{\|y\|=1}\|(\sigma - A)^{-1}y\|^2$ is the square of the minimum modulus of the resolvent.
\item     Let $w_{\Omega}(A)= \sup_{z \in \Omega}(|z|)$ and $w_{W(A)} = \sup_{z \in W(A)}(|z|)$. Then 
    $$
    \left(\frac{1}{w_{\Omega}(A) + 2w_{W(A)}}\right)^2 \leq r(\sigma).
    $$
\end{enumerate}    
\end{Lemma}
\begin{proof}
\begin{enumerate}
    \item Let $\|y\| =1 $ and $y = (\sigma - A)x$. As $ \mu ( \sigma (s) , A )$ is a selfadjoint matrix, we have
    \begin{align}
        &\lambda_{\min} ( \mu ( \sigma (s) , A )) \\
        &= \inf_{\|y\|=1}\langle\mu ( \sigma (s) , A )y,y \rangle \\
        & =\inf_{\|y\| = 1} \left( \frac{1}{\pi} \|x\|^2\Re\left(\left(\sigma - \left\langle A \frac{x}{\|x\|}, \frac{x}{\|x\|} \right)\right\rangle\right)v^*\right) \\
        & \geq \frac{1}{\pi} r(\sigma) \inf_{p \in W(A)} \Re ((\sigma - p)v^*) \\
        &= \frac{1}{\pi} r(\sigma) \left( \Re(\sigma v^*) - h_{W(A)}(v) \right) \\
        &=\frac{1}{\pi} r(\sigma) \left( h_{\Omega}(v) - h_{W(A)}(v) \right),
    \end{align}
    where $\Re(\sigma v^*) =  h_{\Omega}(v) $ because convexity of $\Omega$ and the supporting half plane condition yields $\Re (zv^*) \leq \Re(\sigma v^*)$ for all $z \in \Omega$.
    \item Observe for each $y$ of unit norm, we have $y = (\sigma I - A)(\sigma I - A)^{-1}y$, so the sub-multiplicative property of norms gives
        $$
        \frac{1}{|\sigma| + \|A\|}\leq \frac{1}{\|(\sigma I - A)\|} \leq \| (\sigma I - A)^{-1}y\|.
        $$
        The result follows since $\|A\| \leq 2w_{W(A)} $.
\end{enumerate}    
\end{proof}

We now present the main result of this section which provides a bound for $\gamma(1)$ only in terms of the geometry of $\Omega$ and $W(A)$. In the following theorem $w_{\Omega}(A)= \sup_{z \in \Omega}(|z|)$ and $w_{W(A)} = \sup_{z \in W(A)}(|z|)$.

\begin{Theorem}\label{maingamma}
Let $W(A)$ have non-empty interior and $ W(A) \subseteq \Omega$. Then for $\gamma$ as given in \eqref{gamma},
    $$
    \gamma(1) \leq - \frac{1}{\pi }\left(\frac{1}{w_{\Omega}(A) + 2w_{W(A)}}\right)^2 \min_{\sigma \in \partial \Omega}\rho (\sigma)   \left( |\partial \Omega| -  | \partial W(A)| \right),
    $$
    where $\rho (\sigma) $ is the radius of curvature at a point $ \sigma \in \partial \Omega$.
\end{Theorem}

\begin{proof}
The previous lemma gives
    $$
    \gamma(1) \leq  - \frac{1}{\pi }\left(\frac{1}{w_{\Omega}(A) + 2w_{W(A)}}\right)^2  \int_0^L \left( h_{\Omega}(v) - h_{W(A)}(v) \right) ds .
    $$
    Write $v(s)=e^{i\theta(s)}$ so that for the curvature $\kappa(s)$, we have $\kappa(s)  =  |\theta'(s)|$. Since $W(A)\subseteq\Omega$,
\[
h_\Omega(v(s))-h_{W(A)}(v(s))\geq0.
\]
Therefore, since
$
\min_{\sigma\in\partial\Omega}\rho(\sigma)
=
\frac{1}{\max_{s\in[0,L]}\kappa(s)}
$, we have
\begin{align*}
&\int_0^L \left(h_\Omega(v(s))-h_{W(A)}(v(s))\right)\,ds \\
&\geq
\min_{\sigma\in\partial\Omega}\rho(\sigma)
\int_0^L
\left(h_\Omega(v(s))-h_{W(A)}(v(s))\right)\kappa(s)\,ds.
\end{align*}
By the turning-angle formula,
\[
\int_0^L
\left(
h_\Omega(v(s))-h_{W(A)}(v(s))
\right)\kappa(s)\,ds
=
\int_0^{2\pi}
\left(
h_\Omega(e^{i\theta})-h_{W(A)}(e^{i\theta})
\right)\,d\theta.
\]
Thus
\[
\gamma(1)
\leq
-\frac{1}{\pi}
\left(
\frac{1}{w_\Omega(A)+2w_{W(A)}}
\right)^2
\min_{\sigma\in\partial\Omega}\rho(\sigma)
\int_0^{2\pi}
\left(
h_\Omega(e^{i\theta})-h_{W(A)}(e^{i\theta})
\right)\,d\theta,
\]
and the result follows from Lemma~3.1.
\end{proof}


There are several proofs in the literature (see for example \cite{1root2,okubo1975constants}) showing that if $\Omega$ is a disc with $W(A)\subseteq\Omega$, then $\Omega$ is a $2$-spectral set for $A$. Since $a(\Omega)=0$ when $\Omega$ is a disc and $\sqrt{1-S_{\hat f}(\hat{x})^2}\leq1$, the result above, combined with Theorem~\ref{main}, yields an explicit improvement in the case when $W(A)\subsetneq\Omega$. In particular, it provides a computable constant $K<2$, depending only on the geometry of $W(A)$ and $\Omega$, such that $\Omega$ is a $K$-spectral set for $A$.

\section{Scaled q-Numerical Range as a Spectral Set}\label{4}

In this section for $A \neq \lambda I$ for any $\lambda \in \mathbb{C}$ we consider the scaled $q$-numerical range $\Omega_q:=\frac{W_q(A)}{q}$ for $0<|q| < 1$ as a spectral set. In general $W_q(A)$ cannot be a spectral set since there are matrices with eigenvalues lying outside $W_q(A)$. By contrast, $W(A)\subseteq \Omega_q$ \cite[p. 381]{gau2021numerical}, and since $W(A)$ is a spectral set $\Omega_q$ will also be a spectral set for $A$. Note that $W_q(A)$ (and hence $\Omega_q$) is convex \cite[p. 381]{gau2021numerical} and the spectrum of $A$ lies in the interior of $\Omega_q$ \cite[Theorem 2.10]{li1994generalized}. Since $A$ is an $n\times n$ matrix, $A\neq \lambda I$, and $0<|q|<1$, the boundary of $W_q(A)$, and hence that of $\Omega_q$, is smooth by \cite[Theorem~4]{rajic2005generalized} and thus we may apply the results of the previous sections.

By \cite[Theorem 2.5]{li1994generalized} it follows that 
$\Omega_{q_1} \subseteq \Omega_{q_2}$ 
whenever $0 < |q_2| \le |q_1| < 1$. Thus it is natural to expect the spectral constant for $\Omega_q$ to decrease as $|q|$ decreases. 
Furthermore any non-trivial bound for $\|f(A)\|$ when $f \in \mathcal{A} (\Omega_q(A))$ must depend not only on $q$ but also on a quantitative measure of the non-normality of $A$. The term
\[
\int_0^{2\pi}\sup_{x\in M_\theta}\sqrt{\|Ax\|^2-|\langle Ax,x\rangle|^2}\,d\theta
\]
 where $M_\theta=\arg\max_{\|x\|=1}\Re\langle e^{-i \theta}A x,x\rangle$ plays this role. Moreover, the second ascertain identifies this integral as a variational quantity, describing the initial rate of change of $|\partial \Omega_q|$ as $|q|$ varies. In view of Theorem \ref{main}, the following theorem generalises the result that the numerical range is a
$1+\sqrt{2}$-spectral set \cite{1root2} by providing a bound for the spectral constant of
the scaled $q$-numerical range $\Omega_q$. The bound depends on $q$ and on a
quantity measuring the non-normality of $A$.

\begin{Theorem}
For $0 < |q| < 1$, let $\Omega_q = \frac{W_q(A)}{q}$ and let $W(A)$ have non-empty interior. For $\theta\in[0,2\pi]$, define
\[
M_{\theta}
:=
\arg\max_{\|x\|=1}
\Re\langle e^{-i\theta}Ax,x\rangle,
\]
and set
\[
m(\theta)
:=
\max_{x\in M_{\theta}}
\sqrt{\|Ax\|^2-|\langle Ax,x\rangle|^2}.
\]
Then for $\gamma(1)$ as defined in \eqref{gamma}
\[
\gamma(1)\le
-C\min_{\sigma\in\partial\Omega_q}\rho(\sigma)
\frac{\sqrt{1-|q|^2}}{|q|}
\int_0^{2\pi} m(\theta)\,d\theta,
\]
where
\[
C=\frac1\pi\left(\frac{1}{w_{\Omega_q}(A)+2w_{W(A)}}\right)^2,
\]
and $\rho(\sigma)$ denotes the radius of curvature of $\partial\Omega_q$ at $\sigma$.

Furthermore for $t \in (0, \infty)$ writing $q(t) = \frac{1}{\sqrt{1+t^2}}$ and $\Omega(t) = \Omega_{q(t)}$ we have 
\[
\int_0^{2 \pi} m(\theta)\,d\theta
=
\left.\frac{d}{dt}\right|_{t=0+}\left|\partial\Omega(t) \right|.
\]
\end{Theorem}

\begin{proof}
First observe that $\Omega_q = \Omega_{|q|}$ so it suffices to prove the theorem for $\Omega_{|q|}$. By Theorem~\ref{maingamma} and Lemma~\ref{per}, it is enough to prove that 
\[
\int_0^{2 \pi} \left( h_{\Omega_{|q|}}(e^{ i \theta}) - h_{W(A)}(e^{ i \theta}) \right) d\theta
\geq
\frac{\sqrt{1-|q|^2}}{|q|}
\int_0^{2 \pi} m(\theta)\,d\theta.
\]
Fix $\theta\in[0,2\pi]$ and $x \in \mathbb{C}^n$ with $\|x\|=1$. Then any $y \in \mathbb{C}^n$ with $\|y\|=1$ and $\langle x,y\rangle=|q|$ can be written as
$y=|q|x+\sqrt{1-|q|^2}\,z$ with $z\perp x$ and $\|z\|=1$. Hence
\begin{align}
    &\Re\big(e^{-i\theta}\langle Ax,y\rangle\big)
= |q|\,\Re\langle e^{-i \theta}A x,x\rangle
+\sqrt{1-|q|^2}\,\Re\langle e^{-i \theta}A x,z\rangle \\
& \le |q|\,\Re\langle e^{-i \theta}A x,x\rangle
+ \sqrt{1-|q|^2}\,\|P_{x^\perp}(e^{-i \theta}A x)\|,
\end{align}
where $P_{x^\perp}$ is the projection onto the orthogonal complement of the span of $x$. Clearly equality is attained by choosing $z$ in the direction of $P_{x^\perp}(e^{-i \theta}A x)$ and so
\[
h_{W_{|q|}(A)}(e^{i\theta})
=\sup_{\|x\|=1}\Big(
|q|\,\Re\langle e^{-i \theta}A x,x\rangle
+\sqrt{1-|q|^2}\,\|P_{x^\perp}(e^{-i \theta}A x)\|
\Big),
\]
which means
\begin{equation}\label{star1}
h_{\Omega_{|q|}}(e^{i\theta})
=\sup_{\|x\|=1}\Big(
\Re\langle e^{-i \theta}A x,x\rangle
+\frac{\sqrt{1-|q|^2}}{|q|}\,
\|P_{x^\perp}(e^{-i \theta}A x)\|
\Big).
\end{equation}

Restricting the supremum to $x\in M_\theta$ gives
\[
h_{\Omega_{|q|}}(e^{i\theta})
\geq
h_{W(A)}(e^{i\theta})
+
\frac{\sqrt{1-|q|^2}}{|q|}
\max_{x\in M_\theta}
\|P_{x^\perp}(e^{-i\theta}Ax)\|.
\]
Using
\[
\|P_{x^\perp}(e^{-i\theta}Ax)\|^2
=
\|e^{-i\theta}Ax\|^2-|\langle e^{-i\theta}Ax,x\rangle|^2
=
\|Ax\|^2-|\langle Ax,x\rangle|^2,
\]
we obtain
\[
h_{\Omega_{|q|}}(e^{i\theta})-h_{W(A)}(e^{i\theta})
\geq
\frac{\sqrt{1-|q|^2}}{|q|}\,m(\theta).
\]
Integrating over $[0,2\pi]$, it follows that
\[
\int_0^{2\pi}
\bigl(h_{\Omega_{|q|}}(e^{i\theta})-h_{W(A)}(e^{i\theta})\bigr)\,d\theta
\geq
\frac{\sqrt{1-|q|^2}}{|q|}
\int_0^{2\pi} m(\theta)\,d\theta.
\]
The desired bound for $\gamma(1)$ now follows.

We now prove the furthermore statement. From \eqref{star1}, noting that
$
t=\frac{\sqrt{1-|q|^2}}{|q|},
$
we deduce
\begin{equation}\label{supabove}
\frac{h_{\Omega(t)}(e^{i\theta})-h_{W(A)}(e^{i\theta})}{t} \
=
\sup_{\|x\|=1}
\left(
\frac{\Re\langle e^{-i\theta}Ax,x\rangle-h_{W(A)}(e^{i\theta})}{t}
+
\|P_{x^\perp}(e^{-i\theta}Ax)\|
\right).
\end{equation}

Since $M_\theta$ is compact, choose $x\in M_\theta$ such that
\[
\|P_{x^\perp}(e^{-i\theta}Ax)\|=m(\theta),
\]
so
\begin{equation}\label{t}
\frac{h_{\Omega(t)}(e^{i\theta})-h_{W(A)}(e^{i\theta})}{t}
\geq m(\theta),
\end{equation}
and hence
\begin{equation}\label{eq2}
\liminf_{t\downarrow0}
\frac{h_{\Omega(t)}(e^{i\theta})-h_{W(A)}(e^{i\theta})}{t}
\geq m(\theta).
\end{equation}

We now show 
$$
\limsup_{t\downarrow0}
\frac{h_{\Omega(t)}(e^{i\theta})-h_{W(A)}(e^{i\theta})}{t}
\leq m(\theta).
$$ For each $t>0$ choose a unit vector $x_t$ attaining the supremum in \eqref{supabove}. Let $t_k\downarrow0$ be a sequence along which the corresponding quotient on the left hand side of \eqref{supabove} converges to its $\limsup$. Since the unit sphere in $\mathbb{C}^n$ is compact, after passing to a subsequence we may suppose that $
x_{t_k}\longrightarrow x_0.
$

By \eqref{t},
\[
\Re\langle e^{-i\theta}Ax_{t_k},x_{t_k}\rangle
-h_{W(A)}(e^{i\theta})
+
t_k\|P_{x_{t_k}^\perp}(e^{-i\theta}Ax_{t_k})\|
\geq t_km(\theta).
\]
Since the projection term is uniformly bounded, letting $t_k\to0$ in the above gives $\Re\langle e^{-i\theta}Ax_0,x_0\rangle
\geq
h_{W(A)}(e^{i\theta})$. Clearly by definition we also have 
\[
\Re\langle e^{-i\theta}Ax_{t_k},x_{t_k}\rangle
-h_{W(A)}(e^{i\theta})\leq0,
\]
thus
\[
\Re\langle e^{-i\theta}Ax_0,x_0\rangle
=
h_{W(A)}(e^{i\theta}),
\]
and hence $x_0\in M_\theta$.

Moreover as $x_t$ was chosen maximally in \eqref{supabove} (and noting the first term on the right hand side is always non-positive),
$$
\frac{h_{\Omega(t_k)}(e^{i\theta})-h_{W(A)}(e^{i\theta})}{t_k}
\leq
\|P_{x_{t_k}^\perp}(e^{-i\theta}Ax_{t_k})\|.
$$
Thus letting $k\to\infty$ and using $x_0\in M_\theta$ yields
\begin{equation}\label{eq1}
\limsup_{t\downarrow0}
\frac{h_{\Omega(t)}(e^{i\theta})-h_{W(A)}(e^{i\theta})}{t}
\leq
\|P_{x_0^\perp}(e^{-i\theta}Ax_0)\|
\leq m(\theta).
\end{equation}
Combining \eqref{eq1} and \eqref{eq2} gives

\begin{equation}\label{derriv}
    \left.\frac{d}{dt}\right|_{t=0+}
h_{\Omega(t)}(e^{i\theta})=m(\theta).
\end{equation}

Furthermore, by \eqref{star1},
$$
0\le
\frac{h_{\Omega(t)}(e^{i\theta})-h_{W(A)}(e^{i\theta})}{t}
\le
\sup_{\|x\|=1}\|P_{x^\perp}(e^{-i\theta}Ax)\|
\leq \|A\|.
$$
Thus the dominated convergence theorem allows us to integrate the derivative identity \eqref{derriv}, and by Lemma~\ref{per} we obtain
$$
\int_0^{2\pi}m(\theta)\,d\theta
=
\left.\frac{d}{dt}\right|_{t=0+}
|\partial\Omega(t)|.
$$
\end{proof}

We remark that when $q=1$ the set $\Omega_q$ reduces to the numerical range $W(A)$ and the result recovers the $1+\sqrt{2}$ spectral constant bound.

\begin{Remark}\label{proved}
In the first version of this paper we conjectured that, for every
$n\times n$ matrix $A$, $0<|q|\leq1$ and every polynomial $p$,
\[
\|p(A)\|
\leq
\max\left(
1,
\frac{2|q|}{1+\sqrt{1-|q|^2}}
\right)
\sup_{z\in\Omega_q}|p(z)|.
\]
This gives the classical Crouzeix bound $2$ when $q=1$. Since the first
version of this paper, the conjecture has been proved, and the constant shown
to be sharp, in both \cite{aouichaouioloughlin2026} and \cite{jin2026}.
\end{Remark}

\subsection{An Application to GMRES}

We finish by recording a consequence for GMRES of the sharp scaled
$q$-numerical range estimate recalled in Remark\eqref{proved}. The GMRES algorithm
of Saad and Schultz \cite{saadschultz1986} is a Krylov subspace method for
solving nonsymmetric linear systems $Ax=b$. Given an initial approximation
$x_0$, set $r_0=b-Ax_0$ and
\[
\mathcal K_k(A,r_0)
=
\operatorname{span}\{r_0,Ar_0,\ldots,A^{k-1}r_0\}.
\]
The $k$th GMRES iterate $x_k$ is chosen from
$x_0+\mathcal K_k(A,r_0)$ so as to minimise the Euclidean norm of the
residual $r_k=b-Ax_k$. Equivalently,
\[
\|r_k\|_2
=
\min_{\substack{p\in\mathcal P_k\\ p(0)=1}}
\|p(A)r_0\|_2,
\]
where $\mathcal P_k$ denotes the complex polynomials of degree at most $k$.
Consequently,
\begin{equation}\label{eq:gmresstandard}
\frac{\|r_k\|_2}{\|r_0\|_2}
\leq
\min_{\substack{p\in\mathcal P_k\\ p(0)=1}}
\|p(A)\|.
\end{equation}
Bounds for the right-hand side in terms of the numerical range and, more
generally, spectral sets are well established in the analysis of GMRES (see, for example, \cite{choigreenbaum2015,crouzeix2019spectral, embree2025elman}).

Combining \eqref{eq:gmresstandard} with Remark\eqref{proved} gives,
for every $0<|q|\leq1$,
\begin{equation}\label{eq:gmresq}
\frac{\|r_k\|_2}{\|r_0\|_2}
\leq
\rho_q
\min_{\substack{p\in\mathcal P_k\\ p(0)=1}}
\sup_{z\in\Omega_q}|p(z)|,
\end{equation}
where
\[
\rho_q=
\max\left\{
1,
\frac{2|q|}{1+\sqrt{1-|q|^2}}
\right\}.
\]
For $q=1$, since $\Omega_1=W(A)$ and $\rho_1=2$,
\eqref{eq:gmresq} reduces to the usual numerical-range bound obtained
from Crouzeix's theorem.

The additional parameter $q$ can give a strictly sharper estimate.
For example, consider the non-normal matrix
\[
A=
\begin{pmatrix}
3&1\\
0&5
\end{pmatrix}.
\]
Its numerical range is the ellipse centred at $4$ with major semiaxis
$\frac{\sqrt5}{2}$. For an ellipse centred at $4$, with major axis on the
real line and major semiaxis $a<4$, one has
\[
\min_{\substack{p\in\mathcal P_1\\ p(0)=1}}
\sup_{z\in W(A)}|p(z)|
=
\frac{a}{4};
\]
the minimum is attained by $p(z)=1-\frac{z}{4}$, while the two endpoints
of the major axis give the reverse inequality. It follows that
\[
\min_{\substack{p\in\mathcal P_1\\ p(0)=1}}
\sup_{z\in W(A)}|p(z)|
=
\frac{\sqrt5}{8}.
\]

On the other hand, take $q=\frac{\sqrt3}{2}$. Since $q$-numerical ranges
are rank-one $C$-numerical ranges, \cite[Theorem~4.1]{oloughlinCnumerical}
shows that $\Omega_q$ is the ellipse centred at $4$ with major semiaxis
$
\frac{2\sqrt5+1}{2\sqrt3}.
$
Consequently, noting $
\rho_q=\frac{2\sqrt3}{3}>1
$, \eqref{eq:gmresq} gives
\[
\frac{\|r_1\|_2}{\|r_0\|_2}
\leq
\frac{2\sqrt5+1}{12}
<
\frac{\sqrt5}{4}.
\]
Thus scaled $q$-numerical ranges can yield strictly sharper GMRES
convergence estimates than those obtained from the numerical range alone,
even when the corresponding spectral constant $\rho_q$ is strictly greater
than $1$. More generally, taking the infimum of the right-hand side of
\eqref{eq:gmresq} over $0<q\leq1$ gives a bound no larger than the
classical $q=1$ estimate.

\section*{Acknowledgements}
Ryan O'Loughlin is grateful to EPSRC for financial support (grant - EP/Y008375/1). Jyoti Rani sincerely thanks the Indian Institute of Science Education and Research (IISER) Mohali, India for the financial support
received through an Institute Postdoctoral Fellowship.

\vskip 0.5cm

\noindent The authors have no declarations of interest.

\newpage

\bibliographystyle{plain}
\bibliography{bibliography.bib}
\end{document}